\documentclass[12pt,reqno]{amsart}
\usepackage[english]{babel}
\usepackage{amsmath,amsfonts,amssymb,amsthm,amscd,latexsym}
\usepackage[T1]{fontenc}
\usepackage[utf8]{inputenc}

\usepackage{tikz}
\usetikzlibrary{arrows}
\usepackage{color}
\usepackage{cite}
\usepackage{multirow}
\usepackage{cite}

\newtheorem{theorem}{Theorem}[section]
\newtheorem{lemma}[theorem]{Lemma}

\theoremstyle{definition}

\theoremstyle{remark}
\newtheorem{remark}[theorem]{Remark}
\numberwithin{equation}{section}

\begin{document}

\setcounter{page}{1}

\title[2-local derivations on  infinite-dimensional Lie algebras]{2-local derivations on  infinite-dimensional Lie algebras}

\author[Sh. A. Ayupov, B. B. Yusupov ]{Shavkat Ayupov$^{1,2}$, Baxtiyor Yusupov$^3$}
\address{$^1$ V.I.Romanovskiy Institute of Mathematics\\
  Uzbekistan Academy of Sciences, 81 \\ Mirzo Ulughbek street, 100170  \\
  Tashkent,   Uzbekistan}
\address{$^2$ National University of Uzbekistan, 4, University street, 100174, Tashkent, Uzbekistan}
\email{\textcolor[rgb]{0.00,0.00,0.84}{sh$_{-}$ayupov@mail.ru}}

\address{$^3$ National University of Uzbekistan, 4, University street, 100174, Tashkent, Uzbekistan}
\email{\textcolor[rgb]{0.00,0.00,0.84}{baxtiyor\_yusupov\_93@mail.ru}}



\date{}
\maketitle

\begin{abstract} The present paper is devoted to study 2-local derivations on
infinite-dimensional Lie algebras over a field of characteristic zero. We prove that all 2-local derivations
on the Witt algebra  as well as on the positive Witt algebra  are (global) derivations, and give an example of infinite-dimensional Lie
algebra with a 2-local derivation which is not a derivation.

\\

\end{abstract}

\section{Introduction}

In 1997, \v{S}emrl \cite{Sem} introduced the notion of  2-local derivations and
2-local automorphisms on algebras. Namely, a map \(\Delta : \mathcal{L} \to
\mathcal{L}\) (not necessarily linear) on an algebra \(\mathcal{L}\) is called a \textit{2-local
derivation} if, for every pair of elements \(x,y \in  \mathcal{L},\) there exists a
derivation \(D_{x,y} : \mathcal{L} \to \mathcal{L}\) such that
\(D_{x,y} (x) = \Delta(x)\) and \(D_{x,y}(y) = \Delta(y).\) The notion of 2-local automorphism is given in a similar way.  For a
given algebra \(\mathcal{L}\), the main problem concerning
these notions is to prove that they automatically become a
derivation (respectively, an automorphism) or to give examples of
local and 2-local derivations or automorphisms of \(\mathcal{L},\)
which are not derivations or automorphisms, respectively. Solution
of such problems for finite-dimensional Lie algebras over
algebraically closed field of zero characteristic were obtained in
\cite{AyuKud, AyuKudRak} and \cite{ChenWang}. Namely, in
\cite{AyuKudRak} it is proved that every 2-local derivation on a
semi-simple Lie algebra \(\mathcal{L}\) is a derivation and that
each finite-dimensional nilpotent Lie algebra, with dimension
larger than two admits 2-local derivation which is not a
derivation. Concerning 2-local
automorphism, Chen and Wang in \cite{ChenWang} prove that if \(\mathcal{L},\) is
a simple Lie algebra of type $A_{l},D_{l}$ or $E_{k}, (k = 6, 7,
8)$ over an algebraically closed field of characteristic zero,
then every 2-local automorphism of \(\mathcal{L},\) is an automorphism. Finally,
in \cite{AyuKud} Ayupov and Kudaybergenov generalized this result
of \cite{ChenWang} and proved that every 2-local automorphism of a
finite-dimensional semi-simple Lie algebra over an algebraically
closed field of characteristic zero is an automorphism. Moreover,
they show also that every nilpotent Lie algebra with finite
dimension larger than two admits 2-local automorphisms which is
not an automorphism.

In the present paper we study 2-local derivations on
infinite-dimensional Lie algebras over a field of
characteristic zero.

In Section 2 we give some preliminaries concerning Witt and positive Witt algebras. In Section 3
we give a general form of derivations on the positive Witt algebra. In Section 4
we prove that every 2-local derivations on Witt algebra and on the positive Witt algebra are automatically
derivations. We also show that so-called thin Lie algebras admit 2-local derivations which are not derivations.
\medskip

\section{Preliminaries}

\medskip

In this section we give some necessary definitions and preliminary
results.

A derivation on a Lie algebra \(\mathcal{L}\) is a linear map
$D:\mathcal{L}\rightarrow \mathcal{L}$ which satisfies the Leibniz
law, that is,
$$
D([x,y])=[D(x),y]+[x, D(y)]
$$
for all $x,y\in \mathcal{L}.$ The set of all derivations of
\(\mathcal{L}\) with respect to the commutation operation is a Lie algebra and it
is denoted by $Der(\mathcal{L}).$ For all $a\in \mathcal{L}$,
the map $ad(a)$ on \(\mathcal{L}\) defined as $ad(a)x=[a,x],\ x\in\mathcal{L}$ is a derivation and derivations of this form are called
\textit{inner derivation}. The set of all inner derivations of
\(\mathcal{L}\), denoted $ad(\mathcal{L}),$ is an ideal in
$Der(\mathcal{L}).$

Let $A=\mathbb{C}[x,x^{-1}]$ be the algebra of all Laurent polynomials in one variable over a field of
characteristic zero $\mathbb{F}.$ The Lie algebra of derivations
$$
Der(A)=span\left\{f(x)\frac{d}{dx}:f\in\mathbb{C}\left[x,x^{-1}\right]\right\}
$$
with the Lie bracket is called a \emph{Witt algebra} and denoted by \(W\).
Then \cite{Kac} $W$ is   an  infinite-dimensional simple algebra  which  has the basis $\left\{e_i: e_i=x^{i+1}\frac{d}{dx}, i\in \mathbb{Z}\right\}$ and the
multiplication rule
$$
[e_i, e_j]=(j-i)e_{i+j},\, i,j \in \mathbb{Z}.
$$
We also consider  the infinite-dimensional  positive part \(W^+\) of the Witt algebra. The \emph{positive Witt algebra} \(W^+\) is an infinite-dimensional
Lie algebra \cite{Dim} which has the basis
 $\left\{e_i: e_i=x^{i+1}\frac{d}{dx}, i\in \mathbb{N}\right\}$ and the
multiplication rule
$$
[e_i, e_j]=(j-i)e_{i+j},\, i,j \in \mathbb{N}.
$$

Recall that a map $\Delta: \mathcal{L}\rightarrow \mathcal{L}$
(not liner in general) is called a \textit{2-local derivation} if
for every $x,y\in \mathcal{L},$ there exists a derivation
$D_{x,y}:\mathcal{L}\rightarrow \mathcal{L}$ (depending on $x,y$)
such that $\Delta(x)=D_{x,y}(x)$ and $\Delta(x)=D_{x,y}(y)$.

Since any derivation on the infinite-dimensional Witt algebra
\(W\) is inner \cite{Ikeda}, it follows that for this algebra the above
definition of the 2-local derivation can be reformulated as follows. A map
\(\Delta\) on \(W\) is called a 2-local derivation on \(W,\) if
for any two elements \(x, y \in  W\) there exists an element
\(a_{x, y} \in  W\) (depending on \(x, y\)) such that
\[
\Delta(x) = [a_{x,y} ,x],\,\,\,  \Delta(y) = [a_{x,y}, y].
\]
Henceforth, given a 2-local derivation on \(W,\) the symbol
\(a_{x,y}\) will denote the element from \(W\) satisfying
\(\Delta(x) = [a_{x,y} ,x]\) and \(\Delta(y) = [a_{x,y}, y].\)

\section{Derivations on the positive Witt algebra}

Let us consider the following algebra $W^++\langle e_0\rangle=span\{e_n:n=0,1,2,...\}$ with the  multiplication rule
$$[e_n,e_m]=(m-n)e_{n+m},\ \ \ \ \ n,\ m\geq0.$$

It is clear that $W^+$ is an ideal in $W^++\langle e_0\rangle.$ Hence, any element $a\in W^++\langle e_0\rangle$ defines a spatial derivation $L_a$ on $W^+$ by
the following way:

\begin{equation*}
L_a(x)=[a,x], \ \ x\in W^+.
\end{equation*}

\begin{theorem}\label{thm111} Let $D$ be a derivation on $W^+.$ Then there exists an element $a\in W^++\langle e_0\rangle$ such that
\begin{equation}\label{1212}
D=L_a.
\end{equation}
\end{theorem}
\begin{proof} Let $D(e_1)=\alpha_1e_1+\alpha_2e_2+...+\alpha_ne_n.$
Take an element
$$
a_1=-\alpha_3e_2-\frac{\alpha_4}{2}e_3-...-\frac{\alpha_n}{n-2}e_{n-1}.
$$ Then
$$L_{a_1}(e_1)=[a_1,e_1]=\alpha_3e_3+\alpha_4e_4+...+\alpha_ne_n=D(e_1)-\alpha_1e_1-\alpha_2e_2.$$
Setting $D_1=D-L_{a_1},$ we have a derivation $D_1$ such that $D_1(e_1)=\alpha_1e_1+\alpha_2e_2.$

Now we shall show that $\alpha_2=0.$ Suppose that $D_1(e_2)=\sum\limits_{k\geq1}\beta_ke_k.$
Then
\begin{eqnarray*}\begin{split}
D_1(e_3)&=D_1([e_1,e_2])=[D_1(e_1),e_2]+[e_1,D_1(e_2)]=\\
&= (\alpha_1+\beta_2)e_3+\sum\limits_{k\geq2}k\beta_{k+1}e_{k+2},\\
D_1(e_4)&=\frac{1}{2}D_1([e_1,e_3])=\frac{1}{2}[D_1(e_1),e_3]+\frac{1}{2}[e_1,D_1(e_3)]=\\
&= (2\alpha_1+\beta_2)e_4+\frac{1}{2}(\alpha_2+
6\beta_3)e_5+\frac{1}{2}\sum\limits_{k\geq3}(k^2+k)\beta_{k+1}e_{k+3},\\
D_1(e_5)&=\frac{1}{3}D_1([e_1,e_4])=\frac{1}{3}[D_1(e_1),e_4]+\frac{1}{3}[e_1,D_1(e_4)]=\\
&= (3\alpha_1+\beta_2)e_5+\frac{1}{3}(4\alpha_2+12\beta_3)e_6+\frac{1}{6}\sum\limits_{k\geq3}(k^3+3k^2+2k)\beta_{k+1}e_{k+4},\\
D_1(e_5)&=D_1([e_2,e_3])=[D_1(e_2),e_3]+[e_2,D_1(e_3)]=\\
&=2\beta_1e_4+(\alpha_1+2\beta_2)e_5+4\beta_3e_6+\sum\limits_{k\geq3}(k^2-k+2)\beta_{k+1}e_{k+4}.
\end{split}\end{eqnarray*}\\
Comparing the last two equalities we obtain that
\begin{center}
$\alpha_2=\beta_1=0,$ $\beta_2=2\alpha_1$ and $\beta_{k}=0, \, k\geq4.$
\end{center}
Therefore
\begin{equation*}
D_1(e_1)=\alpha_1e_1\ \ \text{and} \ \ D_1(e_2)=2\alpha_1e_2+\beta_3e_3.
\end{equation*}
Set  $a_2=\alpha_1e_0.$ Then
$$L_{a_2}(e_1)=[\alpha_1e_0,e_1]=\alpha_1e_1.$$
Setting $D_2=D_1-L_{a_2},$ we have a derivation $D_2$  such that $D_2(e_1)=0.$

Now by induction we shall show that
$$
D_2(e_k)=(k-1)\beta_3e_{k+1},\ \text{for all}\  k\geq2.$$\\
We have
\begin{equation*}
D_2(e_2)=D_1(e_2)-L_{a_2}(e_2)=2\alpha_1e_2+\beta_3e_3-2\alpha_1e_2=\beta_3e_3.
\end{equation*}
Assume that
\[
D_2(e_k)=(k-1)\beta_{k}e_{k+1}.
\]
Then
\[
D_2(e_{k+1})=\frac{1}{k-1}D_2([e_1,e_k])=\frac{1}{k-1}[D_2(e_1),e_k]+\frac{1}{k-1}[e_1,D_2(e_{k})]=k \beta_3e_{k+2}.
\]
So
$$
D_2(e_k)=(k-1)\beta_3e_{k+1},\ \text{for all}\  k\geq2.
$$
Take the element $a_3=\beta_3e_1,$ then
$$
L_{a_3}(e_k)=[a_3,e_i]=[\beta_3e_1,e_k]=(k-1)\beta_3e_{k+1}, \ \ k\geq2.$$
This means that $D_2=L_{a_3}.$ Thus
$$
D=D_1+L_{a_1}= D_2+L_{a_1}+L_{a_2}=L_{a_3}+L_{a_2}+L_{a_3}=L_{a},$$
where $a=a_1+a_2+a_3.$
The proof is complete.
\end{proof}
\begin{remark} Let $D=L_a$ be a derivation of the form (\ref{1212}), where $a=-\sum\limits_{i=0}^n\alpha_ie_i.$ Direct computations show that
\begin{equation*}
D(e_j)=\sum\limits_{i=1}^n\alpha_i(j+1-i)e_{i+j-1}, \ \ \ j\geq1.
\end{equation*}
\end{remark}

\section{2-Local derivations on some infinite-dimensional Lie algebras}

Now we shall give the main result concerning 2-local derivations on infinite-dimensional Lie algebras.

\subsection{2-Local derivations on the Witt algebra}

\begin{theorem}\label{thm11} \textit{Let $W$ be the Witt algebra
over a field of characteristic zero. Then any 2-local derivation
on  $W$ is a derivation.}
\end{theorem}
For the proof of this Theorem we need several Lemmas.

\begin{lemma}\label{lem12} \textit{Let $\Delta$ be a 2-local derivation
on \(W\) such that
$$
\Delta(e_{0})=\Delta(e_{1})=0.
$$
Then $\Delta(e_{i})=0$ for all $i\in \mathbb{Z}$}.
\end{lemma}
\begin{proof} Let $i\in\mathbb{Z}$ be a fixed index except
$0,1$. There exists an element $a_{e_{0},e_{i}}\in W$ such that
$$\Delta(e_{0})=[a_{e_{0},e_{i}},e_{0}],\ \ \ \Delta(e_{i})=[a_{e_{0},e_{i}},e_{i}].$$
Then

$$0=\Delta(e_{0})=[a_{e_{0},e_{i}},e_{0}]=\left[\sum\limits_{j\in\mathbb{Z}}\alpha_{j}e_{j},e_{0}\right]=\sum\limits_{j\in\mathbb{Z}}\alpha_{j}je_{j}.$$
Thus  $\alpha_{j}=0$ for all $j\in \mathbb{Z}$ with $j \neq 0.$
This means that $a_{e_{0},e_{i}}=\alpha_{0}e_{0}.$
Thus
$$
\Delta(e_{i})=[a_{e_{0},e_{i}},e_{i}]=[\alpha_{0}e_{0},e_{i}]=i\alpha_{0}e_{i},
$$
i.e.,
\begin{equation}\label{yb1}
\Delta(e_{i})=i\alpha_{0}e_{i}.
\end{equation}
Now take an element $a_{e_{1},e_{i}}\in W$ such that
$$\Delta(e_{1})=[a_{e_{1},e_{i}},e_{1}],\ \ \ \Delta(e_{i})=[a_{e_{1},e_{i}},e_{i}].$$
Then
$$
0=\Delta(e_{1})=[a_{e_{1},e_{i}},e_{1}]=\left[\sum\limits_{j\in\mathbb{Z}}\beta_{j}e_{j},e_{1}\right]=
\sum\limits_{j\in\mathbb{Z}}\beta_{j}(1-j)e_{j},
$$ and therefore  $\beta_{j}=0$ for any $j\in \mathbb{Z},\ j\neq 1.$
This means that $a_{e_{1},e_{i}}=\beta_{1}e_{1}.$
Hence
 $$\Delta(e_{i})=[a_{e_{1},e_{i}},e_{i}]=[\beta_{1}e_{1},e_{i}]=(1-i)\beta_{1}e_{i+1},$$
i.e.,
\begin{equation}\label{yb2}
 \Delta(e_{i})=(1-i)\beta_{1}e_{i+1}.
\end{equation}
Taking into account that $i\neq 0,1,$ and comparing (\ref{yb1}) and (\ref{yb2}) we obtain that $\alpha_{0}=\beta_{1}=0,$ i.e.,
$$\Delta(e_{i})=0.$$
The proof is complete.
\end{proof}
\begin{lemma}\label{lem13}
 \textit{Let $\Delta$ be a 2-local derivation on $W$ such that $\Delta(e_{i})=0$ for all $i\in\mathbb{Z}.$ Then $\Delta\equiv0$}.
\end{lemma}
\begin{proof}  Take an arbitrary element
$x=\sum\limits_{j\in\mathbb{Z}}c_{j}e_{j}.$ Let \(n_x\) be an
index such that $c_i=0$ for all $|i|>n_x.$ Then
$x=\sum\limits_{j=-n_x}^{n_x}c_{j}e_{j}.$  There exists an element
$a_{e_{0},x}\in W$ such that
$$\Delta(e_{0})=[a_{e_{0},x},e_{0}],\ \ \ \Delta(x)=[a_{e_{0},x},x].$$
Since $0=\Delta(e_0)=[a_{e_{0},x},e_{0}],$ it follows that
$$a_{e_{0},x}=\alpha_{0}e_{0}.$$
Therefore
\begin{equation}\label{aabbb}
\Delta(x)=[a_{e_{0},x},x]=\left[\alpha_{0}e_{0},\sum\limits_{j=-n_x}^{n_x}
c_{j}e_{j}\right]=-\sum\limits_{j=-n_x}^{n_x}c_{j}\alpha_{0}je_{j}.
\end{equation}
Now take an index $n>2n_x.$  Let  $a_{e_{n},x}\in W$ be an element
such that
$$
\Delta(e_n)=[a_{e_n,x},e_n],\ \ \ \Delta(x)=[a_{e_n, x}, x].
$$
By lemma \ref{lem12},  $\Delta(e_n)=0.$ Since
$$0=\Delta(e_{n})=[a_{e_{n},x},e_{n}],$$
it follows that
$$a_{e_{n},x}=\alpha_{n}e_{n}.$$
Then
\begin{equation}\label{aabb}
\Delta(x)=[a_{e_{n},x},x]=\sum\limits_{i=-n_{x}}^{n_{x}}[\alpha_{n}e_{n},c_{i}e_{i}]=
\sum\limits_{-n_{x}}^{n_{x}}(n-i)\alpha_{n}c_{i}e_{n+i}.
\end{equation}
Since $n>2n_{x},$  it follows that $n+i>n_{x}$ for all  $i\in
\{-n_{x}, \ldots, n_{x}\}.$ Thus comparing (\ref{aabbb}) and (\ref{aabb}) we
obtain that $\alpha_{n}=0.$ This means that $\Delta(x)=0.$ The
proof is complete.
\end{proof}

Now we are in position to prove Theorem \ref{thm11}.

\textit{Proof of Theorem} \ref{thm11} Let $\Delta$ be a 2-local derivation on \(W\). Take a derivation $D_{e_0,e_1}$ such that
\begin{equation*}
\Delta(e_0)=D_{e_0,e_1}(e_0)\ \ \text{and} \ \ \Delta(e_1)=D_{e_0,e_1}(e_1).
\end{equation*}
Set $\Delta_1=\Delta-D_{e_0,e_1}.$ Then $\Delta_1$ is a 2-local
derivation such that $\Delta_1(e_0)=\Delta_1(e_1)=0.$ By lemma
\ref{lem12}, $\Delta_1(e_i)=0$ for all $i\in\mathbb{Z}.$ By lemma \ref{lem13}, it
follows that $\Delta_1\equiv0.$ Thus $\Delta=D_{e_0,e_1}$ is a
derivation. The proof is complete.
\hfill$\Box$

\subsection{2-Local derivations on the positive  Witt algebra}

\begin{theorem}\label{thm1}
Any 2-local derivation on  $W^+$ is a derivation.
\end{theorem}
For the proof of this Theorem we need several Lemmas.

\begin{lemma}\label{lem1}
Let $\Delta$ be a 2-local derivation
on \(W^+\) such that
$$
\Delta(e_1)=\Delta(e_2)=0.
$$
Then $\Delta(e_{j})=0$ for all $j\in \mathbb{N}$.
\end{lemma}\label{lem2}
 \begin{proof}
By the definition of 2-local derivations, we can find a derivation $D_1$ such that
$$
\Delta(e_1)=D_1(e_1),\ \ \ \ \ \Delta(e_j)=D_1(e_j).$$
Then
$$0=\Delta(e_1)=D_1(e_1)=a_1e_1+\sum\limits_{i\geq3}(2-i)a_ie_i,$$
and therefore
$$a_i=0,\ \ i\neq2.$$
Thus
\begin{equation}\label{a}
\Delta(e_j)=D_1(e_j)=(j-1)a_2e_{j+1},\ \ \ j\geq3.
\end{equation}
Now take a derivation $D_2$ such that
$$
\Delta(e_2)=D_2(e_2)=0,\ \ \ \ \Delta(e_j)=D_2(e_j).
$$
Thus
$$0=\Delta(e_2)=D_2(e_2)=2a_1^{(2)}e_2+a_2^{(2)}e_3+\sum\limits_{i\geq4}(3-i)a_i^{(2)}e_{i+1},$$
i.e.,
$$a_i^{(2)}=0, \ \  i\neq3.$$
This means that
\begin{equation}\label{b}
\Delta(e_j)=(j-2)a_3^{(2)}e_{j+2},\ \ \ j\geq3.
\end{equation}
Comparing \eqref{a} and \eqref{b} we obtain that
$$
(j-1)a_2e_{j+1}=(j-2)a_3^{(2)}e_{j+2},\ \ \ j\geq3.
$$
Thus $a_2=0,$
and therefore $a_i=0$
for all $i.$
So $\Delta(e_j)=0,\ \ j\geq3.$
 The proof is complete.
\end{proof}

\begin{lemma}\label{lem2}
Let $\Delta(e_i)=0$ for all $i\in\mathbb{N}.$ Then $\Delta\equiv0.$
\end{lemma}

\begin{proof} Let  $x=\sum\limits_{k=1}^n x_ke_k$ be a non zero element from $W^+.$
By \eqref{1212} there exists an element $a_{e_1, x}\in W^++\langle e_0\rangle$ such that
$\Delta(e_1)=[a_{e_1, x}, e_1]$ and
$\Delta(x)=[a_{e_1, x}, e_1].$ Since $\Delta(e_1)=0,$ it follows that $a_{e_1, x}=\alpha_1^{(1)}e_1.$
Then
\begin{equation}\label{ab}
\Delta(x)=[a_{e_1, x}, x]=\alpha_1^{(1)}\sum\limits_{k=2}^n (k-1) x_{k}e_{k+1}.
\end{equation}

Now take an number $m$ such that $m> 2n.$
Again by  \eqref{1212} there exists an element $a_{e_m, x}\in W^++\langle e_0\rangle$ such that
$\Delta(e_m)=[a_{e_m, x}, e_m]$ and
$\Delta(x)=[a_{e_m, x}, e_m].$ Since $\Delta(e_m)=0,$ it follows that $a_{e_m, x}=\alpha_{m}^{(m)}e_m.$
Then
\begin{equation}\label{ba}
\Delta(x)=\alpha_{e_m, x}^{(m)} \sum\limits_{k=1}^n (k-m)x_{k}e_{k+m}.
\end{equation}
Taking into account that $m>2n,$ and comparing the right sides of the equalities \eqref{ab} and \eqref{ba}, we obtain that
$\alpha_{e_m, x}^{(m)}=0.$ Thus $\Delta(x)=0.$
The proof is complete.
\end{proof}

Now we are in position to prove Theorem \ref{thm1}.

\textit{Proof of Theorem} \ref{thm1} Let $\Delta$ be a 2-local derivation on $W^+$. Take a derivation $D_{e_1,e_2}$ such that
\begin{equation*}
\Delta(e_1)=D_{e_1,e_2}(e_1)\ \ \text{and} \ \ \Delta(e_2)=D_{e_1,e_2}(e_2).
\end{equation*}
Set $\Delta_1=\Delta-D_{e_1,e_2}.$ Then $\Delta_1$ is a 2-local
derivation such that $\Delta_1(e_1)=\Delta_1(e_2)=0.$ By lemma
\ref{lem1}, $\Delta_1(e_i)=0$ for all $i\in\mathbb{N}.$ By lemma \ref{lem2}, it
follows that $\Delta_1\equiv0.$ Thus $\Delta=D_{e_1,e_2}$ is a
derivation. The proof is complete.
\hfill$\Box$

\subsection{An example of a 2-local derivation which is not a derivation on an infinite-dimensional Lie algebra.}

   Let us consider the following (see \cite{Kha}) so-called {\it thin Lie algebra} $\mathcal{L}$
with a basis \(\{e_n: n\in \mathbb{N}\}\), which is defined by the
following table of multiplications of the basis elements:

 $$[e_1,e_n]=e_{n+1},\ \ \ n\geq 2.$$

 and other products of the basis elements being zero.

\begin{theorem}\label{thm12}
Any derivation $D$ on the algebra  $\mathcal{L}$ has the following
form:
\begin{equation}\begin{split}\label{a232}
D(e_1)&=\sum\limits_{i=1}^n\alpha_ie_i,\\
D(e_2)&=\sum\limits_{i=2}^n\beta_ie_i, \\
D(e_j)&=((j-2)\alpha_1+\beta_2)e_j+\sum\limits_{i=1}^n\beta_{i+2}e_{i+j},\ \ \ j\geq3,
\end{split}\end{equation}
where $\alpha_i, \beta_i\in \mathbb{C},$ $i=1, \ldots, n,$ and $n\in \mathbb{N}.$
\end{theorem}

\begin{proof}  Let $D$ be a derivation on $\mathcal{L}.$  We set
$$D(e_1)=\sum\limits_{i=1}^n\alpha_ie_i,\ \ D(e_2)=\sum\limits_{i=1}^n\beta_ie_i.$$
We have
\begin{equation*}\begin{split}
D(e_3)&=D([e_1,e_2])=[D(e_1),e_2]+[e_1,D(e_2)]=\\
&=\left[\sum\limits_{i=1}^n\alpha_ie_i,e_2\right]+\left[e_1,\sum\limits_{i=1}^n\beta_ie_i\right]=\\
&=(\alpha_1+\beta_2)e_3+\sum\limits_{i=1}^n\beta_{i+2}e_{i+3}.
\end{split}\end{equation*}
Using $[e_2,e_3]=0,$ we have
\begin{equation*}\begin{split}
0=D([e_2,e_3])&=[D(e_2),e_3]+[e_2, D(e_3)]=\left[\sum\limits_{i=1}^n\beta_ie_i,e_3 \right]+\\
&+\left[e_2,(\alpha_1+\beta_2)e_3+\sum\limits_{i=1}^n\beta_{i+2}e_{i+3}\right]=\beta_1e_4,
\end{split}\end{equation*}
Thus $\beta_1=0.$
Further
\begin{equation*}\begin{split}
D(e_4)&=D([e_1,e_3])=[D(e_1),e_3]+[e_1,D(e_3)]=\\
&=\left[\sum\limits_{i=1}^n\alpha_ie_i, e_3\right]+\left[e_1,(\alpha_1+\beta_2)e_3+\sum\limits_{i=1}^n\beta_{i+2}e_{i+3}\right]=\\
&=(2\alpha_1+\beta_2)e_4+\sum\limits_{i=1}^n\beta_{i+2}e_{i+4}.
\end{split}\end{equation*}

With similar arguments applied to the products $[e_1,e_j]=e_{j+1}$ and by the induction on $j$, it is
easy to check that the following identities hold for $j\geq3$:
$$
D(e_j)=((j-2)\alpha_1+\beta_2)e_j+\sum\limits_{i=1}^n\beta_{i+2}e_{i+j},\ \ \ j\geq3.
$$
The proof is complete.
\end{proof}

\begin{theorem}
Let  $\mathcal L$ be the thin Lie algebra. Then $\mathcal{L}$ admits a 2-local derivation which is not a
derivation.
\end{theorem}

\begin{proof}
For $x=\sum\limits_{i=1}^nx_ie_i\in \mathcal{L}$ set
\begin{equation*}\label{yb4}
\Delta(x) = \begin{cases}
0, & \text{if $x_1=0$,}\\
\sum\limits_{i=2}^nx_ie_i, & \text{if $x_1\neq0$}.
\end{cases}
\end{equation*}

We shall show that  $\Delta$ is a 2-local derivation on
\(\mathcal{L}\), which is not a derivation.

Firstly, we show that $\Delta$ is not a derivation. Take the
elements  $x=e_1+e_2$ and $y=-e_1+e_2.$ We have
$$
\Delta(x+y)=\Delta(2e_2)=0$$ and
$$
\Delta(x)+\Delta(y)=\Delta(e_1+e_2)+\Delta(-e_1+e_2)=2e_2.$$ Thus
$$
\Delta(x+y)\neq\Delta(x)+\Delta(y).$$ So, \(\Delta\) is not
additive, and therefore is not a derivation.

Let us consider the linear maps $D_1$ and $D_2$ on \(\mathcal{L}\)
defined as:
\begin{equation}\label{yb5}
D_1(e_n) = \begin{cases}
\sum\limits_{k=2}^m\alpha_ke_k, & \text{if $n=1$,}\\
0, & \text{if $n\geq2$},
\end{cases}
\end{equation}
where $\alpha_k\in \mathbb{C},$ $k=2, \ldots, m,$ and $m\in \mathbb{N}$ and
\begin{equation}\label{yb6}
D_2(e_n) = \begin{cases}
0, & \text{if $n=1$,}\\
e_n, & \text{if $n\geq2$}.
\end{cases}
\end{equation}
By \eqref{a232}, it follows that both $D_1$ and $D_2$ are derivations
on \(\mathcal{L}.\)

For any $x=\sum\limits_{k=1}^{n_x}  x_ke_k, \,
y=\sum\limits_{k=1}^{n_y}  y_ke_k \in \mathcal{L}$ we need to find
a derivation $D=D_{x,y}$ such that
\begin{equation*}
\Delta(x)=D(x)\ \ \text{and} \ \ \Delta(y)=D(y).
\end{equation*}

It suffices to consider the following three cases.

\textbf{Case 1.} Let $x_1=y_1=0.$ In this case, we take $D\equiv
0,$ because $\Delta(x)=\Delta(y)=0.$

\textbf{Case 2.} Let $x_1=0,\ y_1\neq0.$  In this case we take the
derivation $D_1$ of the form \eqref{yb5} with $\displaystyle \alpha_1=0,
\alpha_k=\frac{y_k}{y_1},\) \(2\leq k\leq n_y.\) Then
$$\Delta(x)=0=D_1(x)$$ and
$$\Delta(y)=\sum\limits_{k=2}^{n_y} y_ke_k=y_1\sum\limits_{k=2}^{n_y}\frac{y_k}{y_1}e_k=D_1(y).$$
So, $D$ is a derivation such that $\Delta(x)=D(x),\ \Delta(y)=D(y).$

\textbf{Case 3.} Let $x_1\neq0,\ y_1\neq0.$ In this case we take the derivation $D_2$ of the form~\eqref{yb6}. Then
$$\Delta(x)=\sum\limits_{k=2}^{n_x}x_ke_k=D_2(x)$$
and
$$\Delta(y)=\sum\limits_{k=2}^{n_y}y_ke_k=D_2(y).$$

Therefore in all cases we constructed a derivation on \(\mathcal{L}\) such that $\Delta(x)=D(x),\ \Delta(y)=D(y),$ i.e. $\Delta$ is a 2-local derivation which is not a derivation.
The proof is complete.
\end{proof}

\end{document}